\documentclass{amsart}

\usepackage{amsmath}
\usepackage{amsthm}
\usepackage{amssymb}
\usepackage{enumerate}

\newtheorem{thm}{Theorem}[section]
\newtheorem{theorem}{Theorem}[section]
\newtheorem{prop}[thm]{Proposition}

\newtheorem{lemma}[thm]{Lemma}

\theoremstyle{definition}

\theoremstyle{remark}

\numberwithin{equation}{section}

\newcommand{\R}{\mathbb{R}}  

\everymath{\displaystyle}

\begin{document}


\title{Existence of rotating planet solutions to the Euler-Poisson equations with an inner hard core}


\author{Yilun Wu}
\address{Department of Mathematics, University of Michigan, Ann Arbor, MI 48109}
\email{yilunwu@umich.edu}
\urladdr{http://www-personal.umich.edu/~yilunwu/} 





\begin{abstract}
The Euler-Poisson equations model rotating gaseous stars. Numerous efforts have been made to establish existence and properties of the rotating star solutions. Recent interests in extrasolar planet structures require extension of the model to include a inner rocky core together with its own gravitational potential. In this paper, we discuss various extensions of the classical rotating star results to incorporate a solid core.
\end{abstract}


 \maketitle




\bibliographystyle{acm}

\section{Introduction}

The motion of a rotating Newtonian gaseous star is described by the following compressible Euler-Poisson equations:
\begin{equation}\label{intro: eq: Euler-Poisson system primal}
\begin{cases}
\rho_t + \nabla\cdot(\rho\mathbf{v}) = 0 \\
(\rho\mathbf{v})_t + \nabla\cdot(\rho\mathbf{v}\otimes\mathbf{v}) + \nabla p = -\rho\nabla\phi\\
\Delta \phi = 4\pi \rho
\end{cases}
\end{equation}
Here $\rho$, $p$, $\phi$ and $\mathbf{v}$ are respectively density, pressure, gravitational potential and velocity vector field of the gas comprising a star. The solution to the Poisson equation is not unique. One picks the decaying solution at infinity
\begin{equation}
-\phi = B\rho = \rho * \frac{1}{|\mathbf{x}|} = \int \frac{\rho(\mathbf{y})}{|\mathbf{x}-\mathbf{y}|}~d\mathbf{y}
\end{equation}
according to Newton's law of gravitation, where $*$ defines a convolution. In order to model a rotating star in dynamical equilibrium, one further makes the assumptions that the solution is axisymmetric and time independent. Under these assumptions the first equation in \eqref{intro: eq: Euler-Poisson system primal} is indentically satisfied, whereas the second equation is reduced to 
\begin{equation}\label{intro: eq: vector rotating star}
\frac{\nabla p}{\rho} = -\nabla \phi + r\Omega^2 \mathbf{e}_r.
\end{equation}
Here we are assuming that the velocity vector field $\mathbf{v}$ is given by $\mathbf{v}=r\Omega(r)\mathbf{e}_{\theta}$, where the cylindrical coordinates $(r,\theta,z)$ and orthogonal frame field $\{\mathbf{e}_r, \mathbf{e}_{\theta}, \mathbf{e}_z\}$ are used. To close the system one imposes an equation of state $p=p(\rho)$, sets up the velocity field $\mathbf{v}$, and prescribes the total mass $\int_{\R ^3}\rho ~d\mathbf{x}$. We seek a non-negative axisymmetric solution $\rho$ to \eqref{intro: eq: vector rotating star}. The existence and properties of rotating star solutions to \eqref{intro: eq: vector rotating star} were attained by Auchmuty and Beals \cite{auchmuty1971variational}, Auchmuty \cite{auchmuty1991global}, Caffarelli and Friedman \cite{caffarelli1980shape}, Friedman and Turkington \cite{friedman1980asymptotic,friedman1981existence,friedman1980oblateness}, Li \cite{li1991uniformly}, Chanillo and Li \cite{chanillo1994diameters}, Luo and Smoller \cite{luo2004rotating}, and Luo and Smoller \cite{luo2009existence}.

Recent observations on extrasolar giant gaseous planets have raised fundamental questions about their interior structure and origin. Many of the extrasolar gaseous planets possess unexpectedly small radii, suggesting high metallicity in their composition and possibly the existence of a solid core in the center (Anderson and Adams \cite{anderson2012effects}). Efforts have been made to simulate the evolution of these planets, and evidence for the existence of a solid core has been found (Militzer et al. \cite{militzer2008massive}). Models involving high metallicity and center core have been constructed and examined (Miller et al. \cite{miller2011heavy}, Burrows et al. \cite{burrows2007possible}). As a first model from a mathematical perspective, one could modify the Euler-Poisson equations for rotating stars to include a solid core and its gravitational potential. Let $K$ be an axisymmetric bounded domain in $\mathbb{R}^3$, and $\rho_K$ be a given axisymmetric non-negative function on $K$, indicating the density of the solid core. Let $\phi_K = -\rho_K * \frac{1}{|\mathbf{x}|}$ denote the gravitational potential of $\rho_K$. Then by the $-\phi_K$-modified Euler-Poisson system we mean the following 
\begin{equation}\label{intro: eq: phi_K modified Euler-Poisson system primal}
\begin{cases}
\rho_t + \nabla\cdot(\rho\mathbf{v}) = 0 \\
(\rho\mathbf{v})_t + \nabla\cdot(\rho\mathbf{v}\otimes\mathbf{v}) + \nabla p = -\rho\nabla(\phi + \phi_K) \\
\Delta \phi = 4\pi \rho
\end{cases}
\end{equation}
As is in the case of rotating star solutions, we assume axisymmetry and time independence to reduce the equations as follows.
\begin{equation}\label{intro: eq: vector rotating planet}
\frac{\nabla p}{\rho} = -\nabla (\phi + \phi_K) + r\Omega^2 \mathbf{e}_r.
\end{equation}
The goal of the present paper is to discuss existence and non-existence of solutions to \eqref{intro: eq: vector rotating planet}. 

The idea that led to existence results for \eqref{intro: eq: vector rotating star} is to regard it as a gradient:
\begin{equation}\label{eq: gradient Bernoulli}
\nabla \big( a(\rho) \big)= - \nabla \phi + \nabla J,
\end{equation}
where 
\begin{equation}
a(s) = \int_0^s \frac{p'(t)}{t}dt, \text{ and }J(r) = \int_0^r s\Omega^2(s) ds.
\end{equation}
From \eqref{eq: gradient Bernoulli}, we get
\begin{equation}\label{intro: eq: Bernoulli relation}
a(\rho)= -\phi + J(r) + \lambda 
\end{equation}
for some constant $\lambda$. With prescribed equation of state and angular velocity profile $\Omega(r)$, \eqref{intro: eq: Bernoulli relation} is a single equation for the unknown function $\rho$, although we still don't know the value of $\lambda$ at this stage. 

Now the key idea to solve relation \eqref{intro: eq: Bernoulli relation} is to regard it as the Euler-Lagrange equation of the following energy functional
\begin{equation}
E(\rho) = \int_{\R ^3} \bigg( A(\rho) - \frac{1}{2}\rho B\rho - \rho J \bigg)~d\mathbf{x}
\end{equation}
subject to the constraint 
\begin{equation}
\int_{\R ^3} \rho ~d\mathbf{x}= M.
\end{equation} 
Here 
\begin{equation}
A(s) = \int_0^s a(t)dt
\end{equation}
and
\begin{equation}
B\rho(\mathbf{x}) = \rho * \frac{1}{|\mathbf{x}|} = \int \frac{\rho(\mathbf{y})}{|\mathbf{x}-\mathbf{y}|}~d\mathbf{y}.
\end{equation}
Under this formulation the unknown constant $\lambda$ in \eqref{intro: eq: Bernoulli relation} is naturally realized as a Lagrange multiplier. 

Auchmuty and Beals \cite{auchmuty1971variational} imposed some non-trival decay conditions on $J$ and got the first existence results for rotating stars. However, these conditions excluded constant $\Omega$ and therefore ruled out a large family of physically interesting solutions. Li \cite{li1991uniformly} removed this restriction and was able to obtain an existence result for $J$ with small $L^{\infty}$ norm. This enabled him to prove existence for small constant $\Omega$. The smallness of $\|J\|_{\infty}$ is essential in Li's proof. He also substantiated the smallness requirement by proving a non-existence result for \eqref{intro: eq: Bernoulli relation} with large constant $\Omega$.

Following a similar route, one can write down a scalar equation for the modified Euler-Poisson equations:
\begin{equation}\label{intro: eq: Bernoulli relation rotating planets}
a(\rho)= -\phi -\phi_K + J(r) + \lambda. 
\end{equation}
The corresponding Euler-Lagrange energy is 
\begin{equation}
E(\rho) = \int_{\mathbb{R} ^3 \setminus K} \bigg( A(\rho) - \frac{1}{2}\rho B\rho - \rho J  - B\rho_K \bigg)~d\mathbf{x}.
\end{equation}
Recall that $\rho_K$ is the density of the axisymmetric solid core and is positive. The minus sign in front of the $B\rho_K$ term is slightly surprising. One would expect that the gravitation of an extra center core should somehow cancel out the effect of centrifugal force due to the appearance of $J$, but at least on the energy level, they are of the same sign. In particular, without a smallness assumption on $\rho_K$, Li's proof no longer works. It is physically reasonable to assume slow rotation in order for a solution to exist, but unreasonable to assume smallness of solid core density. 

Furthermore, with a given rotation profile that is not necessarily small, a solution may still exist if the core gravitational pull is sufficiently large. The potential existence of a fast rotating planet with heavy core is a unique phenomenon that is not present in the classical rotating star case.

On the other hand, with a given core density and total mass, there should still be no solution if the rotation is sufficiently large. One may want to look for a non-existence result like the one in \cite{li1991uniformly}. However, the proof in \cite{li1991uniformly} involves a subtle argument based on integral identities, and fails to apply to the case with a solid core. We thus employ a different argument to show non-existence.

\section{Statement of Results}

Let us consider the following axisymmetric equilibrium $\Phi_K$-modified Euler-Poisson equations in $\mathbb{R}^3\setminus K$, for a bounded axisymmetric domain $K$:
\begin{equation}\label{chap3: eq: Euler-Poisson}
\frac{\nabla p}{\rho} = \nabla \big(B\rho+ J + \Phi_K \big),
\end{equation}
which is the gradient of the following equilibrium relation
\begin{equation}\label{chap3: eq: equilibrium eq}
A'(\rho)- B\rho - J -\Phi_K = \lambda.
\end{equation}
Here
\begin{equation}
B\rho(\mathbf{x})=\int_{\mathbb{R}^3\setminus K}\frac{\rho(\mathbf{y})}{|\mathbf{x}-\mathbf{y}|}~d\mathbf{y}
\end{equation}
is the Newtonian potential of $\rho$, 
\begin{equation}
J(r)=\int_0^r s\Omega ^2(s)~ ds,
\end{equation}
where $r=\sqrt{x_1^2 + x_2^2}$, is the potential of centrifugal force, and $\Phi_K$ is the potential generated by the core. We assume 
\begin{equation}\label{chap3: cond: J}
s\Omega^2(s) \text{ is a given non-negative function in } L^1[0,\infty)\cap C[0, \infty). 
\end{equation}
If gravity is the only effect of the core, $\Phi_K$ is given by 
\begin{equation}\label{chap3: eq: Phi_K by gravity}
\Phi_K(\mathbf{x})=B\rho_K(\mathbf{x})=\int_K \frac{\rho_K(\mathbf{y})}{|\mathbf{x}-\mathbf{y}|}~d\mathbf{y}
\end{equation}
where $\rho_K \in L^q(K)$ for some $q>3$ is a given axisymmetric non-negative function on $K$. More generally, $\Phi_K$ is a function satisfying
\begin{equation}\label{chap3: cond: Phi_K}
\Phi_K \in C^1(\mathbb{R}^3)\text{ is positive, axisymmetric, and }\lim_{\textbf{x} \to \infty}\Phi_K(\textbf{x}) = 0,
\end{equation}
and
\begin{align}\label{chap3: cond: Phi_K 2}
&\text{there is a }z_0>0 \text{ such that if } |x_3|>z_0,  ~\Phi_K\text{ is non-increasing as }\\ &|x_3|\text{ increases}. \notag
\end{align}
The equation of state is given by $p = f(\rho)$, where $f$ is a function satisfying
\begin{equation}\label{chap3: cond: f 1}
f(s) \text{ is non-negative, continuous, and strictly increasing for } s>0.
\end{equation}
\begin{equation}\label{chap3: cond: f 2}
\lim_{s\to 0} f(s)s^{-\frac{4}{3}} = 0, \quad \lim_{s\to \infty}f(s)s^{-\frac{4}{3}}=\infty .
\end{equation}
A typical example of such an $f$ would be $f(s)=s^{\gamma}$ for some $\gamma>\frac{4}{3}$.
The internal energy potential $A$ in \eqref{chap3: eq: equilibrium eq} is related to $f$ by
\begin{equation}\label{chap3: eq: define A}
A(s)=s\int_ 0^s \frac{f(t)}{t^2}dt.
\end{equation}

We then have the following

\begin{theorem}\label{chap3: thm: existence of solution with small J}
Given $M>0$, $\Phi_K$ satisfying \eqref{chap3: cond: Phi_K}, and $f$ satisfying \eqref{chap3: cond: f 1} and \eqref{chap3: cond: f 2}, there is an $\epsilon_1 >0$, such that if $\|J\|_{\infty} < \epsilon_1$, there exists a compactly supported axisymmetric continuous function $\rho: \mathbb{R}^3\setminus K\to [0,\infty)$, such that
\begin{enumerate}
\item $\rho$ is differentiable where it is positive, and satisfies the $\Phi_K$-modified Euler-Poisson equation \eqref{chap3: eq: Euler-Poisson} there.\\
\item $\int_{\mathbb{R}^3\setminus K}\rho(\mathbf{x})~d\mathbf{x}=M$.
\end{enumerate} 
\end{theorem}

\begin{theorem}\label{chap3: thm: existence of solution with small Omega}
Given $M>0$, $\Phi_K$ satisfying \eqref{chap3: cond: Phi_K}, and $f$ satisfying \eqref{chap3: cond: f 1} and \eqref{chap3: cond: f 2}, there is an $\epsilon_2 >0$, such that if $\Omega(s)\equiv \Omega < \epsilon_2$, there exists a compactly supported axisymmetric continuous function $\rho: \mathbb{R}^3\setminus K\to [0,\infty)$, such that
\begin{enumerate}
\item $\rho$ is differentiable where it is positive, and satisfies the $\Phi_K$-modified Euler-Poisson equation \eqref{chap3: eq: Euler-Poisson} there.\\
\item $\int_{\mathbb{R}^3\setminus K}\rho(\mathbf{x})~d\mathbf{x}=M$.
\end{enumerate} 
\end{theorem}

Theorem \ref{chap3: thm: existence of solution with small J} and Theorem \ref{chap3: thm: existence of solution with small Omega} establish existence of rotating planet solutions with given mass and core potential for sufficiently small angular velocity profile. 

Furthermore, we have the following

\begin{theorem}\label{chap3: thm: existence of solution with large core gravity, variable angular velocity}
Given $M>0$, $J$ satisfying \eqref{chap3: cond: J}, $f$ satisfying \eqref{chap3: cond: f 1} and \eqref{chap3: cond: f 2}, and $\Phi_K$ satisfying \eqref{chap3: cond: Phi_K}, there is a $\mu_0 >0$, such that if $\mu > \mu_0$, there exists a compactly supported axisymmetric continuous function $\rho: \mathbb{R}^3\setminus C\to [0,\infty)$, such that
\begin{enumerate}
\item $\rho$ is differentiable where it is positive, and satisfies 
the $\mu \Phi_K$-modified Euler-Poisson equations there.\\
\item $\int_{\mathbb{R}^3\setminus K}\rho(\mathbf{x})~d\mathbf{x}=M$.
\end{enumerate} 
\end{theorem}

\begin{theorem}\label{chap3: thm: existence of solution with large core gravity, constant angular velocity}
Given $M>0$, $\Omega(r)\equiv \Omega\geq 0$, $f$ satisfying \eqref{chap3: cond: f 1} and \eqref{chap3: cond: f 2}, and $\Phi_K$ satisfying \eqref{chap3: cond: Phi_K}, there is an $\mu_0 >0$, such that if $\mu > \mu_0$, there exists a compactly supported axisymmetric continuous function $\rho: \mathbb{R}^3\setminus C\to [0,\infty)$, such that
\begin{enumerate}
\item $\rho$ is differentiable where it is positive, and satisfies 
the $\mu \Phi_K$-modified Euler-Poisson equations there.\\
\item $\int_{\mathbb{R}^3\setminus K}\rho(\mathbf{x})~d\mathbf{x}=M$.
\end{enumerate} 
\end{theorem}

Theorem \ref{chap3: thm: existence of solution with large core gravity, variable angular velocity} and Theorem \ref{chap3: thm: existence of solution with large core gravity, constant angular velocity} establish existence of rotating planet solutions with given mass and angular velocity profile for sufficiently large core potential. 

Finally, in order to describe a non-existence theorem for fast constant rotation, we need some further assumptions on the equation of state $f$.
\begin{equation}\label{chap3: cond: f cond nonexistence}
\liminf_{s \to \infty}f(s)s^{-\gamma}>0, \text{ for some }\gamma>\frac{4}{3}.
\end{equation}
$f(s)$ is continuously differentiable for $s>0$ and 
\begin{equation}\label{chap3: cond: f continuous differentiability}
\liminf_{s\to 0}f'(s)s^{-\mu}>0
\end{equation}
for some $\mu>0$. A typical example of such an $f$ is again given by $f(s)=s^{\gamma}$ for some $\gamma >\frac{4}{3}$.

\begin{theorem}\label{chap3: thm: nonexistence of solution}
Suppose $f$ satisfies \eqref{chap3: cond: f 1}, \eqref{chap3: cond: f 2}, \eqref{chap3: cond: f cond nonexistence} and \eqref{chap3: cond: f continuous differentiability}. Let $\Phi_K$ be given by \eqref{chap3: eq: Phi_K by gravity}, and let $M>0$ be given. Also assume that $K$ satisfies the ``no trapping'' condition:
\begin{itemize}
\item[] If $(x,y,z)\in \mathbb{R}^3 \setminus K$, then the half line $(x,y,z)+t(x,y,0)$, ($t\geq 0$) also belongs to $\mathbb{R}^3 \setminus K$.
\end{itemize}
Then there exists an $\Omega_0>0$ such that for $\Omega(r)\equiv \Omega > \Omega_0$, there does not exist a bounded continuous function $\rho: \mathbb{R}^3\setminus C\to [0,\infty)$, such that
\begin{enumerate}
\item $\rho$ satisfies \eqref{chap3: eq: equilibrium eq} where positive.
\item $\int_{\mathbb{R}^3\setminus K}\rho(\mathbf{x})~d\mathbf{x}=M$.
\end{enumerate}
\end{theorem}

Theorems \ref{chap3: thm: existence of solution with small J} and \ref{chap3: thm: existence of solution with small Omega} are proved in section \ref{chap3: slow rotation}. Theorems \ref{chap3: thm: existence of solution with large core gravity, variable angular velocity} and \ref{chap3: thm: existence of solution with large core gravity, constant angular velocity} are proved in section \ref{chap3: dense core}. Theorem \ref{chap3: thm: nonexistence of solution} is proved in section \ref{chap3: non-existence}.

\section{Variational Formulation}
\label{chap3: variational}

We first need a few convolution inequalities. These lemmas turn out to be quite useful for the rotating star existence theory. Their proofs can be found in \cite{auchmuty1971variational}.

\begin{lemma}\label{chap2: lem: L r bound on B rho}
Suppose $\rho\in L^1(\mathbb{R}^3) \cap L^p(\mathbb{R}^3)$, and $1<p\leq \frac{3}{2}$. Then $B\rho \in L^r(\mathbb{R}^3)$ for all $3<r<\frac{3p}{3-2p}$, and
\begin{equation}\label{chap2: ineq: convolution}
\| B \rho \| _r \leq C(\| \rho \|_1 ^b\| \rho \|_p ^{1-b} + \| \rho \|_1 ^c \| \rho \|_p^{1-c})
\end{equation}
for some constant $C$ and $0<b,c<1$ depending on $p$ and $r$. If $p>\frac{3}{2}$, then $B\rho $ is bounded and continuous and satisfies \eqref{chap2: ineq: convolution} with $r=\infty$.
\end{lemma}

\begin{lemma}\label{chap2: lem: bound on gravity potential}
If $\rho \in L^1(\mathbb{R}^3) \cap L^{4/3}(\mathbb{R}^3)$, then
\begin{equation}
\bigg|\int_{\mathbb{R}^3} \rho B\rho ~d\mathbf{x} \bigg|\leq C \bigg( \int_{\mathbb{R}^3}|\rho|^{4/3} ~d\mathbf{x}\bigg) \bigg(\int_{\mathbb{R}^3}|\rho| ~d\mathbf{x}\bigg)^{2/3}.
\end{equation}
\end{lemma}

\begin{lemma}\label{chap2: lem: continuous diff of B rho}
If $\rho\in L^1(\mathbb{R}^3) \cap L^p(\mathbb{R}^3)$ for some $p>3$, then $B\rho$ is continuously differentiable.
\end{lemma}
As \cite{auchmuty1971variational} and \cite{li1991uniformly}, we will solve this problem via a variational approach. Let us consider the energy functional
\begin{equation}\label{chap3: energy functional}
E(\rho)=\int_{\mathbb{R}^3\setminus K}\bigg( A(\rho)(\mathbf{x}) -\frac{1}{2}\rho(\mathbf{x}) B\rho(\mathbf{x}) -\rho(\mathbf{x}) J(\mathbf{x}) -\rho(\mathbf{x})\Phi_K(\mathbf{x})\bigg) ~d\mathbf{x},
\end{equation}
where $A$ is given by \eqref{chap3: eq: define A}, on the space of admissible functions
\begin{equation*}
W=\bigg\{\rho: \mathbb{R}^3\setminus K \to \mathbb{R},  ~\rho \text{ is axisymmetric}, ~\rho \geq 0 \text{ a.e.}, ~\int_{\mathbb{R}^3\setminus K }A(\rho) < \infty, \int_{\mathbb{R}^3\setminus K }\rho = M \bigg\}.
\end{equation*}
We first verify that $E$ is well-defined on $W$. From \eqref{chap3: cond: f 2}, it follows easily that
\begin{equation}\label{chap3: cond: A}
\lim_{s\to 0} A(s)s^{-\frac{4}{3}} = 0, \quad \lim_{s\to \infty}A(s)s^{-\frac{4}{3}}=\infty .
\end{equation}
\eqref{chap3: cond: f 1} and \eqref{chap3: cond: A} imply the existence of a $c>0$ such that
\begin{equation}
A(s)\geq c s^{4/3}
\end{equation}
for $s>1$. Hence
\begin{align}
\int \rho^{4/3} &\leq \frac{1}{c} \int A(\rho) + \int_{\rho<1}\rho^{4/3} \notag \\
& \leq \frac{1}{c} \int A(\rho) + M. \label{chap3: ineq: bound on 4/3 norm}
\end{align}
\eqref{chap3: ineq: bound on 4/3 norm} and lemma \ref{chap2: lem: bound on gravity potential} give the finiteness of the second term in \eqref{chap3: energy functional}. The last two terms in \eqref{chap3: energy functional} are finite because $J$ and $\Phi_K$ are bounded functions. We have shown that $E$ is well-defined on $W$.

The basic assertion is the following:
\begin{prop}\label{chap3: prop: variational principle}
If $\rho$ is a local minimum for $E$ in $W$, then $\rho$ is continuous and is differentiable where it is positive, and satisfies \eqref{chap3: eq: Euler-Poisson} there.
\end{prop}
\begin{proof}
The proof is standard. See \cite{auchmuty1971variational}.
\end{proof}

\section{Existence for Slow Rotation with Fixed Core Density}
\label{chap3: slow rotation}

In the following proof, we will construct a number of bounds $R_{n}$ on the size of the support of the density functions. Without further mentioning, we always assume that $R_{n+1}$ is no less than $R_{n}$. All constants in the following may depend on $M$, $f$, $\|J\|_{\infty}$ and $\Phi_K$. Cartesian coordinates $\textbf{x} = (x_1,x_2,x_3)$ and cylindrical coordinates $(r,\theta,z)$ are used interchangeably. To look for a minimizer of $E$ in $W$, let us first show that $E$ is bounded from below.

\begin{prop}
There is a $C>0$ such that $E(\rho)\geq -C$ for all $\rho \in W$.
\end{prop}
\begin{proof}
By lemma \ref{chap2: lem: bound on gravity potential}, we have
\begin{equation}
E(\rho)\geq \int A(\rho) ~d\textbf{x}- M\|J+\Phi_K\|_{\infty} -\frac{1}{2}cM^{2/3}\int\rho^{4/3}~d\textbf{x}.
\end{equation}
By \eqref{chap3: cond: A}, there is an $s>0$ such that for $\rho>s$, $A(\rho)>\frac{1}{2}cM^{2/3}\rho^{4/3}$. Therefore
\begin{align*}
E(\rho)&\geq \int_{\rho >s} A(\rho)~d\textbf{x} - M\|J+\Phi_K\|_{\infty} -\frac{1}{2}cM^{2/3}\int_{\rho> s}\rho^{4/3}- \frac{1}{2}cM^{2/3}s^{1/3}\int_{\rho < s}\rho ~d\textbf{x} \\
& \geq - M\|J+\Phi_K\|_{\infty}- \frac{1}{2}cM^{5/3}s^{1/3}.
\end{align*}
\end{proof}

Now that we know $E$ is bounded from below, it makes sense to talk about the infimum of $E$. Let 
\begin{equation}
I=\inf_{\rho \in W}E(\rho).
\end{equation}
We will take a sequence of minimizers in bounded balls as a minimizing sequence for $I$. For that purpose, we need to define
\begin{equation}\label{chap3: eq: define W_R}
W_R=\bigg\{\rho \in W ~\big|~ \textbf{Supp}\rho \in S_R, ~0\leq \rho \leq R \text{ a.e.}\bigg\}.
\end{equation}
Here $S_R$ is the ball centered at the origin with radius $R>R_0$ so large that $K$ is contained in $S_R$. As usual we will extend functions in $W_R$ by zero values outside $S_R$, and treat them as functions defined on the whole space if necessary. The next assertion is the starting point of this existence method.
\begin{prop}
There is an $R_0>0$ such that for $R>R_0$, there exists some $\rho_R\in W_R$ which minimizes E:
\begin{equation}
I_R=E(\rho_R)=\inf_{\rho\in W_R}E(\rho).
\end{equation}
\end{prop}
\begin{proof}
The proof is standard. See \cite{auchmuty1971variational} or \cite{li1991uniformly}.
\end{proof}

As in \cite{li1991uniformly}, we can give a uniform $L^{\infty}$ bound on $\rho_R$. 
\begin{lemma}\label{chap3: lem: L infty bound on rho_R}
There is a $C>0$, such that 
\begin{equation}
\|\rho_R\|_{\infty}\leq C
\end{equation}
for all $R\geq R_0$.
\end{lemma}
\begin{proof}
Notice that $\Phi_K\in L^{\infty}(\mathbb{R}^3)$. The proof in this case is basically the same as that in \cite{li1991uniformly}.
\end{proof}

The $L^{\infty}$ bound frees the restriction on $\rho_R$ from above, and therefore implies a variational inequality in one direction:
\begin{lemma}
There is an $R_1>0$, such that for all $R>R_1$, there exists a $\lambda_R$ such that
\begin{align}
A'(\rho_R)- B\rho_R -J -\Phi_K \geq \lambda_R, &\quad \text{ in }B_R,\label{chap3: ineq: Euler Lagrange rho_R}\\
A'(\rho_R)- B\rho_R -J -\Phi_K = \lambda_R, &\quad \text{ where }\rho_R > 0.\label{chap3: eq: Euler Lagrange rho_R}
\end{align}
\end{lemma}
\begin{proof}
See \cite{auchmuty1971variational}.
\end{proof}

\begin{lemma}
There is an $R_2>0$ and $e_1>0$, such that $I_R\leq -e_1$ for all $R>R_2$.
\end{lemma}

\begin{proof}
Let 
\begin{equation}
F(\rho)=\int_{\mathbb{R}^3\setminus K}\bigg( A(\rho)(\mathbf{x}) -\frac{1}{2}\rho(\mathbf{x}) B\rho(\mathbf{x}) \bigg) d\mathbf{x}.
\end{equation}
This is the corresponding energy functional for an Euler-Poisson system with no rotation and a zero density core. The method in \cite{auchmuty1971variational} is fully applicable to this case. We therefore get a compactly supported minimizer $\sigma \in W$ of $F$. Let
\begin{equation}
e_1=-F(\sigma)=-\inf_{\rho \in W}F(\rho).
\end{equation}
$e_1$ is seen to be positive by the following scaling argument: pick a non zero $\rho \in W$ that is bounded and compactly supported in $\mathbb{R}^3\setminus S_{\tilde{R}}$ for some $S_{\tilde{R}}\supset K$. Let
\begin{equation}
\rho_t(\mathbf{x})=t^{-3}\rho(t^{-1}\mathbf{x})
\end{equation}
for $t>1$. We see easily that $\rho_t$ is supported in $\mathbb{R}^3\setminus tB_{\tilde{R}}$, and therefore belongs to $W$. 
\begin{align*}
F(\rho_t)& =\int_{\mathbb{R}^3\setminus tB_{\tilde{R}}}A(\rho_t)-\frac{1}{2}\rho_t B\rho_t \\
& = \int_{\mathbb{R}^3\setminus B_{\tilde{R}}}(t^3A(t^{-3}\rho)-\frac{1}{2}t^{-1}\rho B\rho)\\
& = \int_{\mathbf{Supp} \rho}o(t^{-4}\|\rho\|_{\infty})t^3-t^{-1}\frac{1}{2}\int\rho B\rho\\
& = o(t^{-1})-\Theta(t^{-1}).
\end{align*}
The penultimate step follows from \eqref{chap3: cond: A}. This shows that the minimum of $F$ must be negative. Now let $R_2$ be large enough to contain the support of $\sigma$, then $\sigma \in W_R$ for $R>R_2$, and 
\begin{align*}
E(\rho_R)& \leq E(\sigma) \\
& = \int(A(\sigma)-\frac{1}{2}\sigma B\sigma - J\sigma -\Phi_K\sigma)\\
& \leq \int(A(\sigma)-\frac{1}{2}\sigma B\sigma)\\
& =F(\sigma)\\
& =-e_1.
\end{align*}
\end{proof}

\begin{lemma}\label{chap3: lem: lower bound on mass in unit ball}
Suppose $\|J\|_{\infty}<\frac{e_1}{2M}$. There is an $\epsilon_0>0$ and an $R_2>0$ such that for all $R>R_2$, $\epsilon_R := \sup_{\mathbf{x}\in \mathbb{R}^3}\int_{|\mathbf{x}-\mathbf{y}|<1}\rho_R(\mathbf{y})d\mathbf{y}\geq \epsilon_0$.
\end{lemma}
\begin{proof}
Under the assumption on $\|J\|_{\infty}$
\begin{align*}
& \int \frac{1}{2}\rho_R B\rho_R + \rho_R \Phi_K  \\
=& -E(\rho_R)+\int A(\rho_R)-\rho_R J \\
\geq & \quad e_1 - \|J\|_{\infty}M \\
\geq & \quad \frac{e_1}{2}.
\end{align*}
Therefore either
\begin{equation}\label{chap3: ineq: case 1}
\int \frac{1}{2}\rho_R B\rho_R \geq \frac{e_1}{4},
\end{equation}
or
\begin{equation}\label{chap3: ineq: case 2}
\int \rho_R \Phi_K \geq \frac{e_1}{4}.
\end{equation}
If \eqref{chap3: ineq: case 1} happens, then 
\begin{equation}\label{chap3: ineq: lower bound on B rho_R}
\frac{e_1}{2}\leq \int \rho_R B\rho_R \leq M \|B \rho_R\|_{\infty}.
\end{equation}
Now 
\begin{align*}
B\rho_R(\mathbf{x}) & = \int_{\mathbb{R}^3}\frac{\rho_R(\mathbf{y})}{|\mathbf{x}-\mathbf{y}|}d\mathbf{y}  \\
 & = \int_{|\mathbf{y}-\mathbf{x}|<1}\frac{\rho_R(\mathbf{y})}{|\mathbf{x}-\mathbf{y}|}d\mathbf{y} + \int_{1<|\mathbf{y}-\mathbf{x}|<\tilde{R}}\frac{\rho_R(\mathbf{y})}{|\mathbf{x}-\mathbf{y}|}d\mathbf{y} + \int_{|\mathbf{y}-\mathbf{x}|>\tilde{R}}\frac{\rho_R(\mathbf{y})}{|\mathbf{x}-\mathbf{y}|}d\mathbf{y} \\
 & := B_1 + B_2 + B_3.
\end{align*}
By lemma \ref{chap3: lem: L infty bound on rho_R} and lemma \ref{chap2: lem: L r bound on B rho}, we have
\begin{equation}
B_1 \leq C(\epsilon_R ^b + \epsilon_R ^c)
\end{equation}
for some $0<b,c<1$. The annulus $1<|\mathbf{y}-\mathbf{x}|<\tilde{R}$ can be covered by $C\tilde{R}^3$ balls of radius one, hence
\begin{equation}
B_2 \leq C\tilde{R}^3 \epsilon_R.
\end{equation}
One clearly has
\begin{equation}
B_3 \leq \frac{M}{\tilde{R}}.
\end{equation}
Hence
\begin{equation}\label{chap3: ineq: upper bound on B rho_R}
\|B \rho_R\|_{\infty} \leq C(\epsilon_R ^b + \epsilon_R ^c) +  C\tilde{R}^3 \epsilon_R + \frac{M}{\tilde{R}}.
\end{equation}
Choosing $\tilde{R}$ sufficiently large and comparing \eqref{chap3: ineq: lower bound on B rho_R} with \eqref{chap3: ineq: upper bound on B rho_R}, we see that there must be an $\epsilon_0 >0$ such that $\epsilon_R> \epsilon_0$.
Now let us assume that \eqref{chap3: ineq: case 2} happens. We have
\begin{align*}
& \int \rho_R \Phi_K \\
= & \int_{|\mathbf{x}|>\tilde{R}}\rho_R(\mathbf{x}) \Phi_K(\mathbf{x})d\mathbf{x} + \int_{|\mathbf{x}|<\tilde{R}}\rho_R(\mathbf{x}) \Phi_K(\mathbf{x})d\mathbf{x} \\
:= & B_1 + B_2.
\end{align*}
By \eqref{chap3: cond: Phi_K}, we can choose $\tilde{R}$ so large that $\Phi_K(\mathbf{x})\leq \frac{e_1}{8M}$ when $|\mathbf{x}|>\tilde{R}$. Then
\begin{equation}
B_1\leq \frac{e_1}{8}.
\end{equation}
The ball $|\mathbf{x}|<\tilde{R}$ can be covered by $C\tilde{R}^3$ balls of radius one, hence
\begin{equation}
B_2 \leq C\tilde{R}^3 \epsilon_R.
\end{equation}
Therefore
\begin{equation}\label{chap3: ineq: upper bound on int rho Phi}
\int \rho_R \Phi_K \leq \frac{e_1}{8} + C\tilde{R}^3 \epsilon_R.
\end{equation}
Comparing \eqref{chap3: ineq: case 2} with \eqref{chap3: ineq: upper bound on int rho Phi}, we again see that such an $\epsilon_0$ exists. 
\end{proof}

\begin{lemma}\label{chap3: lem: upper bound on R_a}
There is an $R_a>0$ such that if
\begin{equation}
\int_{|\mathbf{y}-\mathbf{x}|<1}\rho_R(\mathbf{y})d\mathbf{y}\geq \frac{\epsilon_0}{2},
\end{equation}
then $r(\mathbf{x})\leq R_a$. Here $r(\mathbf{x})=\sqrt{x_1^2 +x_2^2}$.
\end{lemma}
\begin{proof}
Assume $|r(\mathbf{x})|>\tilde{R}+1$ where $S_{\tilde{R}}\supset K$. By the axisymmetry of $\rho_R$,
\begin{align*}
Cr(\mathbf{x})\frac{\epsilon_0}{2} & \leq \int_{T}\rho_R \leq M, \\
r(\mathbf{x}) &\leq \frac{2M}{C\epsilon_0}.
\end{align*} 
Here $T$ is the torus obtained from rotating the the ball $|\mathbf{y}-\mathbf{x}|<1$ around the $z$-axis.
\end{proof}

\begin{lemma}\label{chap3: lem: uniform negative upper bound on lambda_R}
Suppose $\|J\|_{\infty}\leq \frac{e_1}{2M}$. There is an $R_3>R_a$ and an $e_2>0$ such that $\lambda_R\leq -e_2$ for all $R>R_2$.
\end{lemma}
\begin{proof}
By lemma \ref{chap3: lem: lower bound on mass in unit ball}, for $R>R_2$ there is an $\mathbf{x}_R$ such that 
\begin{equation}
\int_{|\mathbf{y}-\mathbf{x}_R|<1}\rho_R(\mathbf{y})d\mathbf{y}\geq \frac{\epsilon_0}{2}.
\end{equation}
By lemma \ref{chap3: lem: upper bound on R_a}, $r(\mathbf{x}_R)<R_a$. Let $\mathbf{x}_0$ be on the $z$-axis such that $z(\mathbf{x}_0)=z(\mathbf{x}_R)$. Let $B(\mathbf{x}_0,R_3)$ be the ball centered at $\mathbf{x}_0$ with radius $R_3>R_a$ to be determined. When $R>R_3$, the volume of  the set $B(\mathbf{x}_0,R_3) \cap B_R$ is of order $R_3^3$. There must exist a point $\mathbf{x}\in B(\mathbf{x}_0,R_3)\cap B_R$ such that 
\begin{equation}
\rho_R(\mathbf{x})\leq \frac{CM}{R_3^3}
\end{equation}
for some constant $C>0$.
Clearly
\begin{equation}
|\mathbf{x}-\mathbf{x}_R|\leq |\mathbf{x}-\mathbf{x}_0|+|\mathbf{x}_R-\mathbf{x}_0|\leq 2R_3.
\end{equation}
Hence
\begin{equation}
B\rho_R(\mathbf{x})\geq \int_{|\mathbf{y}-\mathbf{x}_R|<1}\frac{\rho_R(\mathbf{y})}{|\mathbf{x}-\mathbf{y}|}d\mathbf{y}\geq \frac{1}{2R_3+1}\frac{\epsilon_0}{2}.
\end{equation}
By \eqref{chap3: ineq: Euler Lagrange rho_R}, 
\begin{align}\label{chap3: ineq: upper bound on lambda_R}
\lambda_R\leq A'(\frac{CM}{R_3^3})-\frac{1}{2R_3+1}\frac{\epsilon_0}{2}
\end{align}
Notice that \eqref{chap3: cond: f 2} implies
\begin{equation}
\lim_{s\to 0}\frac{A'(s)}{s^{1/3}}=0.
\end{equation}
Hence \eqref{chap3: ineq: upper bound on lambda_R} implies
\begin{align}\label{chap3: ineq: upper bound on lambda_R final}
\lambda_R \leq o(R_3^{-1})-\Theta(R_3^{-1}).
\end{align}
Pick $R_3$ so large that the right hand side of \eqref{chap3: ineq: upper bound on lambda_R final} becomes negative, and call that $-e_2$.
\end{proof}

\begin{lemma}
Suppose $\|J\|_{\infty}\leq \min\bigg\{\frac{e_1}{2M},\frac{e_2}{2}\bigg\}$, then
\begin{equation}\label{chap3: ineq: lower bound on B rho_R + Phi_K}
B\rho_R + \Phi_K \geq \frac{e_2}{2} \quad \text{where }\rho_R>0
\end{equation}
for $R>R_3$.
\end{lemma}
\begin{proof}
By \eqref{chap3: eq: Euler Lagrange rho_R} and lemma \ref{chap3: lem: uniform negative upper bound on lambda_R}, we have
\begin{equation}
A'(\rho_R)-B\rho_R-J-\Phi_K = \lambda_R\leq -e_2
\end{equation}
when $\rho_R>0$.
\end{proof}

\begin{lemma}\label{chap3: lem: r bound on the support of rho}
Suppose $\|J\|_{\infty}\leq \min\bigg\{\frac{e_1}{2M},\frac{e_2}{2}\bigg\}$. There exists an $R_4>0$ such that $\rho_R(\mathbf{x})=0$ if $R>r(\mathbf{x})>R_4$.
\end{lemma}
\begin{proof}
We have
\begin{align*}
B\rho_R(\mathbf{x})& =\int \frac{\rho_R(\mathbf{y})}{|\mathbf{x}-\mathbf{y}|}d\mathbf{y}\\
& = \int_{|\mathbf{x}-\mathbf{y}|<1} \frac{\rho_R(\mathbf{y})}{|\mathbf{x}-\mathbf{y}|}d\mathbf{y} + \int_{1<|\mathbf{x}-\mathbf{y}|<\tilde{R}} \frac{\rho_R(\mathbf{y})}{|\mathbf{x}-\mathbf{y}|}d\mathbf{y} + \int_{|\mathbf{x}-\mathbf{y}|>\tilde{R}} \frac{\rho_R(\mathbf{y})}{|\mathbf{x}-\mathbf{y}|}d\mathbf{y} \\
& := B_1 + B_2 + B_3.
\end{align*}
Clearly
\begin{equation}
B_3\leq \frac{M}{\tilde{R}}.
\end{equation}
We choose $\tilde{R}>\frac{12M}{e_2}$, so that
\begin{equation}
B_3 < \frac{e_2}{12}.
\end{equation}
By lemma \ref{chap2: lem: L r bound on B rho}, 
\begin{align}
B_1\leq c_0 \bigg( \bigg(\int_{|\mathbf{x}-\mathbf{y}|<1} \rho_R(\mathbf{y})d\mathbf{y}\bigg)^b + \bigg(\int_{|\mathbf{x}-\mathbf{y}|<1} \rho_R(\mathbf{y})d\mathbf{y} \bigg)^c \bigg)
\end{align}
for some $0<b,c<1$. By requiring $R>r(\mathbf{x})>R_4$ to be large enough, we have
\begin{equation}
B_1 \leq c_0\bigg( \bigg(\frac{CM}{R_4} \bigg)^b + \bigg(\frac{CM}{R_4} \bigg)^c\bigg) < \frac{e_2}{12}
\end{equation}
by axisymmetry, just like in lemma \ref{chap3: lem: upper bound on R_a}. The annulus $1<|\mathbf{x}-\mathbf{y}|<\tilde{R}$ can be covered by $C\tilde{R}^3$ balls of radius 1. Again by axisymmetry, we have
\begin{equation}
B_2\leq \frac{C\tilde{R}^3M}{R_4-\tilde{R}}<\frac{e_2}{12},
\end{equation}
provided $R_4$ is chosen to be sufficiently large. Therefore
\begin{equation}
B\rho_R(\mathbf{x})= B_1 + B_2 + B_3 < \frac{e_2}{4}
\end{equation}
if $R>r(\mathbf{x})>R_4$. Enlarge $R_4$ if necessary so that $\Phi_K(\mathbf{x})<\frac{e_2}{4}$ when $r(\mathbf{x})>R_4$. We get
\begin{equation}\label{chap3: ineq: upper bound on B rho_R + Phi_K}
B\rho_R(\mathbf{x}) + \Phi_K(\mathbf{x}) < \frac{e_2}{4} + \frac{e_2}{4} =\frac{e_2}{2}
\end{equation}
when $R>r(\mathbf{x})>R_4$. Comparing \eqref{chap3: ineq: upper bound on B rho_R + Phi_K} with \eqref{chap3: ineq: lower bound on B rho_R + Phi_K}, we see that the assertion is true.
\end{proof}

\begin{lemma}\label{chap3: lem: pre z bound on the support of rho}
Suppose $\|J\|_{\infty}\leq \min\bigg\{\frac{e_1}{2M},\frac{e_2}{2}\bigg\}$. There exist $R_5>0$, $\delta>0$ and $r>0$ such that if $R>z(\mathbf{x})>R_5$, and if
\begin{equation}
\int_{|z(\mathbf{x})-z_0|<r}\rho_R(\mathbf{x})d\mathbf{x}<\delta,
\end{equation}
then $\rho(\mathbf{x})=0$ for $|z(\mathbf{x})-z_0|<1$.
\end{lemma}
\begin{proof}
Suppose $r>2$. If $|z(\mathbf{x})-z_0|<1$, $\text{dist}\big(\mathbf{x},\{\mathbf{y} ~\big|~ |z(\mathbf{y})-z_0|>r\}\big) > r-1$. Just like in lemma \ref{chap3: lem: r bound on the support of rho}, we have
\begin{align*}
B\rho_R(\mathbf{x}) & = \int_{|z(\mathbf{y})-z_0|<r} \frac{\rho_R(\mathbf{y})}{|\mathbf{x}-\mathbf{y}|}d\mathbf{y} + \int_{|z(\mathbf{y})-z_0|>r} \frac{\rho_R(\mathbf{y})}{|\mathbf{x}-\mathbf{y}|}d\mathbf{y} \\
& \leq C(\delta^b +\delta ^c)+\frac{M}{r-1} \\
& < \frac{e_2}{4}
\end{align*}
by choosing $\delta$ small and $r$ large. Furthermore $\Phi_K(\mathbf{x})<\frac{e_2}{4}$ if $z(\mathbf{x})>R_5$ is sufficiently large. These imply
\begin{equation}
B\rho_R(\mathbf{x}) + \Phi_K(\mathbf{x}) < \frac{e_2}{2}.
\end{equation}
The assertion follows again from a comparison with \eqref{chap3: ineq: lower bound on B rho_R + Phi_K}.
\end{proof}

\begin{lemma}\label{chap3: lem: z bound on the support of rho}
Suppose $\|J\|_{\infty}\leq \min\bigg\{\frac{e_1}{2M},\frac{e_2}{2}\bigg\}$. There is an $R_6>0$ such that $\rho_R(\mathbf{x})=0$ if $R>z(\mathbf{x})>R_6$.
\end{lemma}
\begin{proof}
Let $Z_n=\{x : |z(x)-2n|<1\}$, $n=\pm ([R_5]+1), \pm ([R_5]+2), \dots$, and let $Z_n'=\big\{x ~\big|~ |z(x)-2n|<r\big\}$. By lemma \ref{chap3: lem: pre z bound on the support of rho}, if $\rho_R$ is not identically zero on a $Z_n$, then $\int_{Z_n'}\rho_R\geq \delta$. Let $m$ be the number of such $n$'s. Since each point in $\mathbb{R}^3$ is covered by at most $r$ different $Z_n'$'s, $m\delta \leq rM$. Also such $Z_n$'s must be contiguous, if they lie in the region $|z|>z_0+2$ for $z_0$ given in \eqref{chap3: cond: Phi_K 2}. Otherwise there would be an ``empty'' $Z_n$ below a ``non-empty'' half space. If one slides the whole ``non-empty'' half space down by two units to create a new $\rho_R'$, $\int A(\rho_R)-J \rho_R = \int A(\rho_R') - J\rho_R'$, but $\int -\frac{1}{2}\rho_R B\rho_R- \rho_R \Phi_K > \int -\frac{1}{2}\rho_R' B\rho_R'- \rho_R' \Phi_K $. This implies $E(\rho_R)>E(\rho_R')$, but $\rho_R'\in W_R$, a contradiction. Now pick $R_6 > 2\bigg( [R_5]+\frac{rM}{\delta}\bigg)+z_0 +3$. The proof is complete.
\end{proof}

We are now in a position to prove theorem \ref{chap3: thm: existence of solution with small J} and theorem \ref{chap3: thm: existence of solution with small Omega}.

\begin{proof}[Proof of theorem \ref{chap3: thm: existence of solution with small J}]
Let $\epsilon_1=\min\bigg\{\frac{e_1}{2M},\frac{e_2}{2}\bigg\}$. From lemma \ref{chap3: lem: r bound on the support of rho} and lemma \ref{chap3: lem: z bound on the support of rho}, we see that $\rho_R=\rho_{R_7}$ when $R>R_7 := \sqrt{2}R_6$. Since $\Phi_K \in L^{\infty}(\mathbb{R}^3)$, a similar argument as in \cite{auchmuty1971variational} shows that $\rho=\rho_{R_7}$ minimizes $E$ in $W$. By proposition \ref{chap3: prop: variational principle}, $\rho$ solves \eqref{chap3: eq: Euler-Poisson} and has the stated properties.
\end{proof}

\begin{proof}[Proof of theorem \ref{chap3: thm: existence of solution with small Omega}]
Let $\epsilon_2=\frac{\sqrt{\epsilon_1}}{R_7}$, and let $\tilde{J}(r)\in C^{\infty}(0,\infty)$ be an increasing function such that
\begin{equation}
\tilde{J}(r)=
\begin{cases}
\frac{1}{2}\Omega^2r^2 &\quad \text{if }r \leq R_7, \\
\Omega^2 R_7^2 &\quad \text{if } r \geq 2R_7.
\end{cases}
\end{equation}
If $\Omega < \epsilon_2$, we have $\|\tilde{J}\|< \epsilon_1$, hence by theorem \ref{chap3: thm: existence of solution with small J}, there is a solution $\rho$ to \eqref{chap3: eq: Euler-Poisson} where $J$ is replaced by $\tilde{J}$, supported in $S_{R_7}$. Clearly such a $\rho$ also solves \eqref{chap3: eq: Euler-Poisson} with the original $J$, and has the stated properties.
\end{proof}

\section{Existence for Fast Rotation with Heavy Core Density}
\label{chap3: dense core}

In this section, we give proofs to theorem \ref{chap3: thm: existence of solution with large core gravity, variable angular velocity} and theorem \ref{chap3: thm: existence of solution with large core gravity, constant angular velocity}. That corresponds to establishing existence of minimizer of
\begin{equation}
E_{\mu}(\rho)=\int_{\mathbb{R}^3\setminus C}\bigg( A(\rho)(\mathbf{x}) -\frac{1}{2}\rho(\mathbf{x}) B\rho(\mathbf{x}) -\rho(\mathbf{x}) J(\mathbf{x}) -\mu \rho(\mathbf{x})\Phi_K(\mathbf{x})\bigg) d\mathbf{x}
\end{equation}
for large enough $\mu$. We will omit an argument in the proof if it runs parallel to the proof in the previous section.

As before, $E_{\mu}$ is bounded from below on $W$ and has an infimum which we denote by $I_{\mu}$. If we pick 
\begin{equation}\label{chap3: eq: defining W_R without upper bound}
W_R=\bigg\{\rho \in W ~\big|~ \textbf{Supp}\rho \in S_R, ~\rho \geq 0 \text{ a.e.}\bigg\}.
\end{equation}
$E_{\mu}$ will also attain its infimum $I_{\mu, R}$ on each $W_R$. We still denote the minimizers by $\rho_R$. It is understood that $\rho_R$ implicitly depends on $\mu$. Comparing \eqref{chap3: eq: define W_R} with \eqref{chap3: eq: defining W_R without upper bound}, we see that the $L^{\infty}$ bound on $W_R$ (namely, the $\leq R$ constraint) is removed. This is to allow large $\rho_R$ on $B_R$. As we will see later, the $L^{\infty}$ bound of $\rho_R$ depends on $\mu$ and $J$. For that purpose, we start by modifying the bound on $\|\rho\|_{4/3}$.

\begin{lemma}\label{chap3: lem: upper bound on rho^4/3}
Let $\rho_R$ be a minimizer of $E_{\mu}$ in $W_R$, and assume that $B_{R_0}$ contains the core $K$. There is a constant $C$ depending only on $f$, $M$, $J$ and $\Phi_K$ such that
\begin{equation}
\int \rho_R^{\frac{4}{3}}~d\mathbf{x} \leq C( 1+ \mu ) 
\end{equation}
for all $R>R_0$.
\end{lemma}
\begin{proof}
Let $\rho_0$ be some fixed function in $W_{R_0}$. For $R>R_0$,
\begin{align*}
 & \int \bigg( A(\rho_0) - \rho_0 J-\frac{1}{2}\rho_0 B\rho_0 -\mu \rho_0 \Phi_K \bigg)~d\textbf{x}\\
 \geq & \int \bigg( A(\rho_R) - \rho_R J-\frac{1}{2}\rho_{R}B\rho_{R}-\mu \rho_{R}\Phi_K \bigg) ~d\textbf{x}\\
 \geq & \int \bigg( A(\rho_R) - \rho_R(J+\mu \Phi_K) \bigg) ~d\textbf{x}-CM^{\frac{2}{3}}\int \rho_R^{\frac{4}{3}}~d\textbf{x}.
\end{align*}
The last step follows from lemma \ref{chap2: lem: bound on gravity potential}. By condition \eqref{chap3: cond: f 2}, there is an $s_1>0$ such that
\begin{equation}
A(s)s^{-\frac{4}{3}}>2CM^{\frac{2}{3}}
\end{equation}
for $s>s_1$. Therefore
\begin{align*}
\tilde{C} &= \int \bigg( A(\rho_0) - \rho_0 J-\frac{1}{2}\rho_0 B\rho_0 -\mu \rho_0 \Phi_K \bigg)~d\textbf{x} \\
& \geq \int_{\rho_R \leq s_1}A(\rho_R) ~d\textbf{x}+\int_{\rho_R> s_1}A(\rho_R) ~d\textbf{x} -M(\|J\|_{\infty} +\mu \|\Phi_K\|_{\infty}) \\
& \quad \quad \quad -CM^{\frac{2}{3}} s_1^{\frac{1}{3}}M -\int_{\rho_R>s_1}\frac{1}{2}A(\rho_R) ~d\textbf{x}\\
& \geq \frac{1}{2} \int A(\rho_R)~d\textbf{x}- C' (1+ \mu).  
\end{align*}
Or,
\begin{equation}
\int A(\rho_R) ~d\textbf{x}\leq C(M,s_1) (1+\mu ).
\end{equation}
Notice that we have
\begin{align*}
\int \rho_R^{\frac{4}{3}} ~d\textbf{x}& = \int_{\rho_R \leq s_1} \rho_R^{\frac{4}{3}}~d\textbf{x}+\int_{\rho_R > s_1} \rho_R^{\frac{4}{3}}~d\textbf{x} \\
& \leq s_1^{\frac{1}{3}}M + \frac{1}{2CM^{\frac{2}{3}}}\int_{\rho_R >s_1}A(\rho_R)~d\textbf{x}\\
& \leq C(M,s_1) \bigg(1 + \int A(\rho_R)~d\textbf{x}\bigg).
\end{align*}
The assertion is now apparent.
\end{proof}

Now let us give an $L^{\infty}$ bound on $\rho_R$. It is crucial to make the power of $\mu$ as low as possible.
\begin{lemma}\label{chap3: lem: l infty bound on rho_R for heavy core}
There is an $R_1>0$ and a constant $C$ depending on $f$, $M$, $J$ and $\Phi_K$ such that 
\begin{equation}
\|\rho_R\|_{\infty}\leq C(1+\mu)
\end{equation} 
for $R>R_1$.
\end{lemma}
\begin{proof}
Let $E_R=\big\{\mathbf{x}\in \mathbb{R}^3 \setminus K ~\big| ~ \rho_R(\mathbf{x})>10M\big\}$, $F_n=\big\{\mathbf{x}\in \mathbb{R}^3 \setminus K ~\big| ~ 10M< \rho_R(\mathbf{x})<n\}$ for $n$ large. It is easy to see that the Lebesgue measure $|E_R|<\frac{1}{10}$. Choose $D\subset B_R \setminus E_R$ such that $|D|=1$. This is possible if we choose some $R_1>\max\{R_0,10\}$. Now let $\gamma_1=\frac{4}{3}$ and $\alpha_1=\frac{5\gamma_1-6}{3}-\epsilon=\frac{2}{9}-\epsilon$ for some very small $\epsilon>0$ to be determined later. Now define
\begin{equation}
v_1=\begin{cases}
-\rho_R^{1+\alpha_1} \quad & \text{on } F_n \\
\int_{F_n}\rho_R^{1+\alpha_1} \quad & \text{on } D\\
0 \quad & \text{otherwise}
\end{cases}
\end{equation}
One sees that $\rho_R+tv_1 \in W_R$ for $t>0$ sufficiently small. Since $\rho_R$ is a minimizer of $E_{\mu}$ in $W_R$, we have $\lim_{t\to 0^+}\frac{E_{\mu}(\rho_R + tv_1)-E_{\mu}(\rho_R)}{t}\geq 0$. Calculating the limit, we get
\begin{equation}
\int (A'(\rho_R)-J-B\rho_R -\mu \Phi_K )v_1 \geq 0,
\end{equation}
from which it follows that
\begin{equation}
-\int_{F_n} v_1 A'(\rho_R)\leq \int_{D}v_1 A'(\rho_R) - \int_{F_n}v_1(J+\mu \Phi_K) -\int_{F_n} v_1 B\rho_R. 
\end{equation}
Condition \eqref{chap3: cond: f 2} on $f$ implies that $A'(s)\geq C_1 \rho^{\frac{1}{3}}$ for $s> 10M$. Therefore
\begin{equation}
-\int_{F_n}v_1 A'(\rho_R) \geq \frac{1}{C_1}\int_{F_n} \rho_R^{\frac{4}{3}+\alpha_1}.
\end{equation}
Furthermore,
\begin{align*}
\int_{D}v_1 A'(\rho_R) & \leq A'(10 M)\int_{F_n} \rho_R^{1+\alpha_1}, \\
-\int_{F_n}v_1(J+\mu \Phi_K) & \leq (\|J\|_{\infty} +\mu \|\Phi_K\|_{\infty})\int_{F_n}\rho_R^{1+\alpha_1}, \\
-\int_{F_n}v_1 B\rho_R & \leq \|\rho_R^{1+\alpha_1}\|_{\frac{3\gamma_1}{5\gamma_1- 3- 3\epsilon}} \|B\rho_R\|_{(\frac{1}{\gamma_1}-\frac{2}{3}+\frac{\epsilon}{\gamma_1})^{-1}} \\
& = \|\rho_R\|_{(1+\alpha_1)\frac{3\gamma_1}{5\gamma_1- 3- 3\epsilon}}^{1+\alpha_1} \|B\rho_R\|_{(\frac{1}{\gamma_1}-\frac{2}{3}+\frac{\epsilon}{\gamma_1})^{-1}} \\
& = \|\rho_R\|_{\gamma_1}^{1+\alpha_1} \|B\rho_R\|_{(\frac{1}{\gamma_1}-\frac{2}{3}+\frac{\epsilon}{\gamma_1})^{-1}} \\
& \leq C \|\rho_R\|_{\gamma_1}^{2+\alpha_1}.
\end{align*}
Here the last step follows from lemma \ref{chap2: lem: L r bound on B rho}. Now 
\begin{align*}
& \int_{F_n} \rho_R^{\frac{4}{3}+\alpha_1} \\
\leq & C_1(A'(10M)+\|J\|_{\infty} +\mu \|\Phi_K\|_{\infty})\int_{F_n}\rho_R^{1+\alpha_1} + C \|\rho_R\|_{\gamma_1}^{2+\alpha_1} \\
\leq & C_2(1+\mu)\|\rho_R\|_{1+\alpha_1}^{1+\alpha_1}+ C \|\rho_R\|_{\gamma_1}^{2+\alpha_1} .
\end{align*}
Since $1+\alpha_1<\gamma_1$, by the interpolation inequality for $L^p$ spaces,
\begin{equation}
\|\rho_R\|_{1+\alpha_1}\leq C(M)\|\rho_R\|_{\gamma_1}^{\frac{4\alpha_1}{1+\alpha_1}}.
\end{equation}
Hence 
\begin{align*}
& \int_{F_n} \rho_R^{\frac{4}{3}+\alpha_1} \\
\leq & C_3(1+\mu)\|\rho_R\|_{\gamma_1}^{4\alpha_1}+ C \|\rho_R\|_{\gamma_1}^{2+\alpha_1} \\
\leq & C_3(1+\mu)\bigg(\int \rho_R^{\frac{4}{3}}\bigg)^{3\alpha_1} + C\bigg(\int \rho_R^{\frac{4}{3}}\bigg)^{\frac{3}{4}(2+\alpha_1)}\\
\leq & C_4(1+\mu)^{1+3\alpha_1}+C_4(1+\mu)^{\frac{3}{4}(2+\alpha_1)} \\
\leq & 2C_4(1+\mu)^{\frac{5}{3}}.
\end{align*}
Lemma \ref{chap3: lem: upper bound on rho^4/3} is needed for the penultimate step, and the last step follows from the choice of $\alpha_1$. Now let $n$ tend to infinity. Since the $F_n$'s increase to $E_R$, one gets
\begin{equation}
\int_{E_R} \rho_R^{\frac{4}{3}+\alpha_1} \leq 2C_4(1+\mu)^{\frac{5}{3}}.
\end{equation}
\begin{align*}
\int \rho_R^{\frac{4}{3}+\alpha_1} & = \int_{E_R} \rho_R^{\frac{4}{3}+\alpha_1} + \int_{\rho_R\leq 10M} \rho_R^{\frac{4}{3}+\alpha_1} \\
& \leq 2C_4(1+\mu)^{\frac{5}{3}} + (10M)^{\frac{4}{3}+\alpha_1-1}M \\
& \leq C_5(1+\mu)^{\frac{5}{3}}.
\end{align*}
Or,
\begin{equation}\label{chap3: ineq: bound on 4/3 + alpha_1 norm}
\|\rho_R\|_{\frac{4}{3}+\alpha_1}\leq C_5(1+\mu)^{\frac{5}{4+3\alpha_1}}.
\end{equation}
Here we assumed that we had chosen $\epsilon$ so small that 
\begin{equation}\label{chap3: ineq: alpha_1,1}
\frac{4}{3}+\alpha_1 = \frac{14}{9}-\epsilon > \frac{3}{2}.
\end{equation} 
Let $b_1(\mathbf{x})=\frac{1}{|\mathbf{x}|}\chi_{S_1}(\mathbf{x})$ and $b_2(\mathbf{x})=\frac{1}{|\mathbf{x}|}-b_1(\mathbf{x})$. We have $B\rho_R=\rho_R*b_1 + \rho_R * b_2$.
\begin{equation}\label{chap3: ineq: bound on b2 part}
\|\rho_R * b_2\|_{\infty}\leq \|b_2\|_{\infty}\|\rho_R\|_{1}\leq C.
\end{equation}
Now let us pick some $p$ between $1$ and $2$. Assume that we have chosen $\epsilon$ so small that the following is true
\begin{equation}\label{chap3: ineq: alpha_1,2}
\frac{1}{1-\frac{p}{5}(1+3\alpha_1)}>\frac{3}{2}.
\end{equation}
Notice that since $\alpha_1=\frac{2}{9}-\epsilon<\frac{2}{9}$, $1-\frac{p}{5}(1+3\alpha_1)>1-\frac{p}{3}>0$. \eqref{chap3: ineq: alpha_1,2} is equivalent to $\alpha_1>\frac{1}{3}(\frac{5}{3p}-1)$. Since the right hand side is less than $\frac{1}{3}(\frac{5}{3}-1)=\frac{2}{9}$, this is possible. Now choose $q$ satisfying $q>\frac{3}{2}$, $q<\frac{4}{3}+\alpha_1$, $q<\frac{1}{1-\frac{p}{5}(1+3\alpha_1)}$. That this is possible follows from \eqref{chap3: ineq: alpha_1,1} and \eqref{chap3: ineq: alpha_1,2}. Since $b_1\in L^{q'}$ for $1\leq q' < 3$, 
\begin{align*}
\|\rho_R * b_1\|_{\infty} &\leq \|b_1\|_{q'}\|\rho_R\|_{q} \\
& \leq C(M)\|\rho_R\|_{\frac{4}{3}+\alpha_1}^a
\end{align*}
where $a=\frac{1-\frac{1}{q}}{1-\frac{1}{\frac{4}{3}+\alpha_1}}$, by the interpolation inequality for $L^p$ spaces. Now it follows from \eqref{chap3: ineq: bound on 4/3 + alpha_1 norm} that
\begin{equation}
\|\rho_R * b_1\|_{\infty} \leq C_6(1+\mu)^{\frac{5a}{4+3\alpha_1}}.
\end{equation}
Combining this with \eqref{chap3: ineq: bound on b2 part}, we get
\begin{equation}
\|B\rho_R\|_{\infty}\leq C_7(1+\mu)^{\frac{5a}{4+3\alpha_1}}.
\end{equation}
Let us calculate the exponent:
\begin{align*}
& \frac{5a}{4+3\alpha_1} \\
= & \frac{1-\frac{1}{q}}{1-\frac{1}{\frac{4}{3}+\alpha_1}} \frac{5}{4+3\alpha_1} \\
= & \frac{5(1-\frac{1}{q})}{1+3\alpha_1} < p
\end{align*}
by the choice of $q$. Therefore
\begin{equation}\label{chap3: ineq: power p bound on Brho}
\|B\rho_R\|_{\infty}\leq C_7(1+\mu)^p.
\end{equation}
Now if $p\geq 3$, the same inequality is obviously true since it is already true for smaller exponents. Now let $\alpha_{l+1}=\alpha_{l}+\frac{1}{3}$. Define
\begin{equation}
v_l=\begin{cases}
-\rho_R^{1+\alpha_l} \quad & \text{on } F_n \\
\int_{F_n}\rho_R^{1+\alpha_l} \quad & \text{on } D\\
0 \quad & \text{otherwise}
\end{cases}
\end{equation}
and repeat the previous argment, only this time using the better estimate \eqref{chap3: ineq: power p bound on Brho}. That gives us
\begin{equation}
\int \rho_R^{\frac{4}{3}+\alpha_l}\leq C_8(1+\mu)^p\int \rho_R^{1+\alpha_l},
\end{equation}
or,
\begin{align*}
\int \rho_R^{1+\alpha_{l+1}} & \leq C_8(1+\mu)^p\int \rho_R^{1+\alpha_l} \\
& \leq (C_8(1+\mu)^p)^l \int \rho_R^{1+\alpha_1} \\
& \leq (C_9(1+\mu)^p)^l \int \rho_R^{\frac{4}{3}} \\
& \leq (C_9(1+\mu)^p)^l C (1+\mu).
\end{align*}
Therefore
\begin{align*}
\|\rho_R\|_{\infty} &= \lim_{l \to \infty} \|\rho_R\|_{1+\alpha_{l+1}}\\
& \leq \lim_{l\to \infty} (C_9(1+\mu)^p)^{\frac{l}{l+1}} (C (1+\mu))^{\frac{1}{l+1}}\\
& \leq C_9(1+\mu)^p.
\end{align*}
We now use this better bound on $\rho_R$ to estimate
\begin{align*}
\|\rho_R*b_1\|_{\infty} & \leq \|b_1\|_{2}\|\rho_R\|_2 \\
& \leq C(M)\|\rho_R\|_{\infty}^{\frac{1}{2}} \\
& \leq C_{10}(1+\mu)^{\frac{p}{2}}.
\end{align*}
Hence
\begin{equation}\label{chap3: ineq: heavy core sublinear growth in mu of Brho_R}
\|B\rho_R\|_{\infty}\leq C_{11}(1+\mu)^{\frac{p}{2}}.
\end{equation}
Since we have chosen $p<2$, this grows at most linearly in $\mu$. We can now repeat the previous bootstrap argument with this better estimate on $\|B\rho_R\|_{\infty}$. One gets
\begin{equation}
\int \rho_R^{\frac{4}{3}+\alpha_l}\leq C_{12}(1+\mu)\int \rho_R^{1+\alpha_l},
\end{equation}
and the assertion of the lemma follows.
\end{proof}

$\rho_R$ still satisfies variational equations like \eqref{chap3: ineq: Euler Lagrange rho_R} and \eqref{chap3: eq: Euler Lagrange rho_R} for $R>R_1$. From here on, we will construct a series of bounds $R_n$ on the support of $\rho_R$, and a series of lower bounds $\mu_n$ for $\mu$. Let us emphasize from the beginning that although the $\mu_n$'s depend on $f$, $M$, $\Phi_K$ and $J$, the $R_n$'s are independent of $J$ and $\mu$. Also we always take $R_{n+1}\geq R_n$ and $\mu_{n+1}\geq \mu_n$.

\begin{lemma}\label{chap3: lem: heavey core bound on lambda_R}
There is an $R_2>0$ and a $\tilde{K}>0$ such that $\lambda_R\leq 1-\mu \tilde{K}$ for $R>R_2$.
\end{lemma}
\begin{proof}
One first observes that if $R>R_2>R_1$, there must be a point $\mathbf{x}\in S_{R_2}$ such that
\begin{equation}
\rho_R(\mathbf{x})\leq \frac{M}{\frac{4}{3}\pi R_2^3}.
\end{equation}
By \eqref{chap3: ineq: Euler Lagrange rho_R},
\begin{equation}\label{chap3: ineq: heavy core bound on lambda_R}
\lambda_R \leq A'\bigg(\frac{M}{\frac{4}{3}\pi R_2^3}\bigg)-\mu \Phi_K(\mathbf{x}).
\end{equation}
\eqref{chap3: cond: f 2} implies that 
\begin{equation}
\lim_{s\to 0}\frac{A'(s)}{s^{1/3}}=0.
\end{equation}
Hence 
\begin{equation}
A'\bigg(\frac{M}{\frac{4}{3}\pi R_2^3}\bigg)=o(R_2^{-1}).
\end{equation}
Pick $R_2$ large enough to make $A'\bigg(\frac{M}{\frac{4}{3}\pi R_2^3}\bigg)<1$, and let $\tilde{K}=\inf_{B_{R_2}}\Phi_K>0$. By \eqref{chap3: ineq: heavy core bound on lambda_R},
\begin{equation}
\lambda_R \leq 1-\mu \tilde{K}.
\end{equation}
\end{proof}

\begin{lemma}
There is a $\mu_2>0$ such that if $\mu>\mu_2$ and $R>R_2$,
\begin{equation}\label{chap3: ineq: extra bound}
B\rho_R + \mu \Phi_K \geq \frac{\mu \tilde{K}}{2} \text{ where } \rho_R>0.
\end{equation}
\end{lemma}
\begin{proof}
By \eqref{chap3: eq: Euler Lagrange rho_R} and lemma \ref{chap3: lem: heavey core bound on lambda_R},
\begin{equation}
A'(\rho_R)-B\rho_R - J -\mu \Phi_K = \lambda_R \leq 1-\mu \tilde{K}
\end{equation}
where $\rho_R>0$. Hence
\begin{equation}
B\rho_R +\mu \Phi_K \geq \mu \tilde{K}-1-J.
\end{equation}
Pick $\mu_2>\frac{2(1+\|J\|_{\infty})}{\tilde{K}}$ to get the result.
\end{proof}

\begin{lemma}\label{chap3: lem: heavy core bound on support rho_R}
There is a $\mu_3>0$ and an $R_3>0$ such that $\rho_R(\mathbf{x})=0$ if $R>|\mathbf{x}|>R_3$ and $\mu>\mu_3$.
\end{lemma}
\begin{proof}
We only need to prove $B\rho_R + \mu \Phi_K < \frac{\mu \tilde{K}}{2}$ in view of \eqref{chap3: ineq: extra bound}. By \eqref{chap3: ineq: heavy core sublinear growth in mu of Brho_R}, $\|B\rho_R\|_{\infty}\leq C(1+\mu)^a$ for some $0<a<1$. We may choose $\mu_3$ so large that $\frac{C(1+\mu)^a}{\mu}<\frac{\tilde{K}}{4}$ when $\mu>\mu_3$, and $R_3$ so large that $\Phi_K(\mathbf{x})<\frac{\tilde{K}}{4}$ when $|\mathbf{x}|>R_3$. The lemma then follows.
\end{proof}

\begin{proof}[Proof of theorem \ref{chap3: thm: existence of solution with large core gravity, variable angular velocity} and theorem \ref{chap3: thm: existence of solution with large core gravity, constant angular velocity}] The argument goes exactly as before. For the constant angular velocity case just notice that the $R_3$ in lemma \ref{chap3: lem: heavy core bound on support rho_R} only depends on $f$, $M$, $\Phi_K$ and not on $J$ and $\mu$, so we can construct a smooth increasing function
\begin{equation}
J(r)=
\begin{cases}
\frac{1}{2}\Omega^2 r^2 \quad &\text{if }r<R_3, \\
\Omega^2 R_3^2 \quad & \text{if } r>2R_3,
\end{cases}
\end{equation}
and find a $\mu_0$ such that a solution exists and is supported in $S_{R_3}$ if $\mu>\mu_0$.
\end{proof}

\section{Non-existence for Fast Rotation with Fixed Core Density}
\label{chap3: non-existence}

We now show that a solution does not exist for large enough constant rotation if the core potential $\Phi_K$ is given by the gravity of a density function $\rho_K$. Let us start with a few estimates.

\begin{lemma}\label{chap3: lem: nonexist L infty bound on B rho}
Let $\rho\in L^{\infty}(\mathbb{R}^3)$ be a non-negative function such that $\int \rho=M$, then there is a $C>0$ such that
\begin{equation}
\|B\rho\|_{\infty}\leq C M^{\frac{2}{3}}\|\rho\|_{\infty}^{\frac{1}{3}}.
\end{equation}
\end{lemma}
\begin{proof}
See \cite{friedman2010variational}.
\end{proof}

\begin{lemma}\label{chap3: lem: nonexist L infty bound on B rho_r}
Let $\rho\in L^{\infty}$ be a nonnegative function supported in the infinite cylinder $x_1^2+x_2^2\leq d^2$. Then there is a $C>0$, such that for $x_1^2+x_2^2 \leq d^2$, 
\begin{equation}\label{chap3: ineq: upper bound on the r derivative of B rho}
|(B\rho) _r(\mathbf{x})|\leq C\|\rho\|_{\infty} \bigg(d+\sqrt{x_1^2+x_2^2}\bigg).
\end{equation}
Here the subscript $r$ denotes directional derivative in the cylindrical radial direction, even if the function under consideration is not axisymmetric.
\end{lemma}
\begin{proof}
Without loss of generality, we may assume $x_1\geq 0,x_2=0,x_3=0$. 
\begin{align}
|(B\rho)_r(x_1,0,0)| &\leq \bigg| \int_{\textbf{supp}\rho}\frac{\rho(x_1',x_2',x_3')(x_1-x_1')}{\sqrt{(x_1'-x_1)^2+x_2'^2+x_3'^2}^3}dx_1'dx_2'dx_3' \bigg| \notag \\
& \leq \int_{\textbf{supp} \rho \cap \{x_1'<x_1\}}\frac{\|\rho\|_{\infty}(x_1-x_1')}{\sqrt{(x_1'-x_1)^2+x_2'^2+x_3'^2}^3}dx_1'dx_2'dx_3' \notag \\
& \leq \|\rho\|_{\infty} \int_{-d<x_1'<x_1}\frac{x_1-x_1'}{\sqrt{(x_1'-x_1)^2+x_2'^2+x_3'^2}^3}dx_1'dx_2'dx_3' \notag \\
& = C\|\rho\|_{\infty} (d+x_1). \notag
\end{align}
The last equality follows either from a direct calculation or a simple application of the divergence theorem.
\end{proof}

\begin{lemma}\label{chap3: lem: Holder and Lip continuity of Phi_K}
Let $l=\sup \big\{|x_3|~\big| ~ (x_1,x_2,x_3)\in K\big\}+1$, $Z=\big\{(x_1,x_2,x_3)~\big|~ |x_3|\leq l\big\}$. Then
$\Phi_K|_Z=B\rho_K|_Z \in C^{1,\frac{3}{q}}(\bar{Z})$, $\Phi_K|_{\mathbb{R}^3\setminus Z}=B\rho_c|_{\mathbb{R}^3 \setminus Z} \in C^{1,1}(\overline{ \mathbb{R}^3 \setminus Z})$.
\end{lemma}
\begin{proof}
We first estimate $\Phi_K|_{\mathbb{R}^3\setminus Z}$:
\begin{equation*}
\Phi_K(x_1,x_2,x_3)=\int_{\textbf{Supp}\rho_K}\frac{\rho_K(x_1',x_2',x_3')}{\sqrt{(x_1-x_1')^2+(x_2-x_2')^2+(x_3-x_3')^2}}dx_1'dx_2'dx_3'.
\end{equation*}
Since $(x_1,x_2,x_3)$ is bounded away from $\textbf{Supp}\rho_K$, we can differentiate under the integral sign and see that
\begin{align*}
|D\Phi_K(x,y,z)|&\leq C \int_{\textbf{Supp}\rho_K}\frac{\rho_K(x_1',x_2',x_3')}{\sqrt{(x_1-x_1')^2+(x_2-x_2')^2+(x_3 - x_3')^2}^2}dx_1'dx_2'dx_3' \\
&\leq C\int_{\textbf{Supp}\rho_K}\rho_K(x_1',x_2',x_3')dx_1'dx_2'dx_3' \\
&\leq C\|\rho_K\|_1\\
&\leq \tilde{C}\|\rho_K\|_{q}^q.
\end{align*} 
In the above inequalities, the second line is because $|x_3-x_3'|\geq 1$, the last line is because $\textbf{Supp}\rho_K$ is compact. We can give a similar estimate for $D^2\Phi_K$, therefore $\Phi_K|_{\mathbb{R}^3\setminus Z} \in C^{1,1}(\mathbb{R}^3\setminus Z)$. As for $\Phi_K|_Z$, the Lipschitz continuity of the first derivative in a neighborhood of $\infty$ follows in the same way as above, whereas the H\"older continuity of the first derivative in a neighborhood of $\textbf{Supp}~\rho_K$ follows from the standard Calderon-Zygmund inequality and the Sobolev embedding theorem.
\end{proof}

From now on, we assume $\Omega$ is at least $1$ and use cylindrical coordinates $(r,\theta,z)$. Let us suppose, contrary to the assertion of theorem \ref{chap3: thm: nonexistence of solution}, that there is such a $\rho$ satisfying all the properties stated.
\begin{lemma}\label{chap3: lem: bound on d}
$d=\sup\big\{r~\big|~(r,\theta,z)\in \bf{Supp} \rho\big\}<\infty$.
\end{lemma}
\begin{proof}
By \eqref{chap3: eq: equilibrium eq},
\begin{equation}
\frac{1}{2}r^2\leq \frac{1}{2}\Omega^2 r^2 \leq A'(\rho)-B\rho - \Phi_K -\lambda \leq A'(\rho)-\lambda.
\end{equation}
We know that $A'(s)=\int_0^s \frac{f(t)}{t^2}dt+\frac{f(s)}{s}$. It follows from \eqref{chap3: cond: f 1} and \eqref{chap3: cond: f 2} that $A'(\rho)\in L^{\infty}$ if $\rho$ is.
\end{proof}

By the expression of $A'(s)$ in the proof, we see that $A'(\rho)>0$ iff $\rho>0$.

\begin{lemma}\label{chap3: lem: size of lambda}
$\lambda \leq -\frac{1}{2}\Omega^2 d^2$.
\end{lemma}
\begin{proof}
Pick a sequence $(r_n,\theta_n,z_n)$ such that $\rho(r_n,\theta_n,z_n)>0$, and $r_n\to d$. We claim that $A'(\rho)(r_n,\theta_n,z_n)\to 0$. If not, a subsequence will be bounded away from zero. Without loss of generality, we still call that subsequence $A'(\rho)(r_n,\theta_n, z_n)$. By the no trapping condition, $A'(\rho)(r,\theta_n,z_n)$ is defined for all $r>r_n$, in particular we have $A'(\rho)(d,\theta_n,z_n)=0$. By Rolle's theorem there is an $r_n^*$ between $r_n$ and $d$ such that $A'(\rho)(r_n^*,\theta_n, z_n)>0$ and $(A'(\rho))_r(r_n^*,\theta_n, z_n) \to -\infty$. By \eqref{chap3: eq: equilibrium eq} and the smoothing effect of $B$, $A'(\rho)$ is differentiable when positive. Differentiating \eqref{chap3: eq: equilibrium eq}, we get
\begin{equation}\label{chap3: eq: force balance in the r direction}
(A'(\rho))_r-\Omega^2 r -(B\rho)_r-(\Phi_c)_r=0.
\end{equation}
We see a contradiction if we evaluate this expression at $(r_n^*,\theta_n, z_n)$: the first term goes to $-\infty$ while the last three terms are bounded by lemma \ref{chap3: lem: bound on d}, lemma \ref{chap3: lem: nonexist L infty bound on B rho_r} and lemma \ref{chap3: lem: Holder and Lip continuity of Phi_K} respectively. Now evaluate \eqref{chap3: eq: equilibrium eq} at $(r_n,\theta_n,z_n)$. By the limit of $A'(\rho)(r_n,\theta_n,z_n)$ and the positivity of $B\rho$ and $\Phi_K$, we get the desired result.
\end{proof}

\begin{lemma}\label{chap3: lem: nonexistence l infinity bound on rho}
There is a constant $C_1>0$, depending on $\Phi_K$, $f$ and $M$, such that $\|\rho\|_{\infty}\leq C_1$. 
\end{lemma}
\begin{proof}
By lemma \ref{chap3: lem: size of lambda}, 
\begin{equation}
A'(\rho)\leq B\rho +\Phi_K.
\end{equation}
By \eqref{chap3: cond: f cond nonexistence}, there is a $C>0$ such that if $s>C$, 
\begin{equation}
Cs ^{\gamma-1}\leq A'(s).
\end{equation}
Hence either $\rho<C$ or $C\rho^{\gamma -1}\leq B\rho +\Phi_K$. Therefore
\begin{equation}
C\|\rho\|_{\infty}^{\gamma -1}\leq C^{\gamma}+\|\Phi_K\|_{\infty}+CM^{\frac{2}{3}}\|\rho\|_{\infty}^{\frac{1}{3}}.
\end{equation}
The last term follows from lemma \ref{chap3: lem: nonexist L infty bound on B rho}. Here we have taken the liberty of using the same constant $C$. Now take $\epsilon>0$ so small that $\epsilon M^{\frac{2}{3}}<\frac{1}{2}$. Since $\gamma-1>\frac{1}{3}$, we have 
\begin{equation}
\|\rho\|_{\infty}^{\frac{1}{3}}\leq \epsilon \|\rho\|_{\infty}^{\gamma -1} + C(\epsilon).
\end{equation}
It follows that
\begin{align*}
C\|\rho\|_{\infty}^{\gamma -1} &\leq C^{\gamma}+\|\Phi_K\|_{\infty}+\frac{1}{2}C\|\rho\|_{\infty}^{\gamma -1} + C(M)\\
\frac{1}{2}C\|\rho\|_{\infty}^{\gamma -1} &\leq C^{\gamma}+\|\Phi_K\|_{\infty} + C(M).
\end{align*}
The assertion now follows from the fact that $\Phi_K \in L^{\infty}(\mathbb{R}^3)$.
\end{proof}

\begin{lemma}\label{chap3: lem: nonexist uniform bound on d}
There is an $\Omega_1>0$ and $0<d_0<\frac{1}{4}$ such that if $\Omega>\Omega_1$, then $d<d_0$.
\end{lemma}
\begin{proof}
Pick an $(r,\theta,z)$ such that $\rho(r,\theta,z)>0$, $r>\frac{d}{2}$. Then there is an $r^*$ between $r$ and $d$ such that $(A'(\rho))_r(r^*, \theta, z)\leq 0$. Evaluating \eqref{chap3: eq: force balance in the r direction} at this point, we have
\begin{align}
& \Omega^2\frac{d}{2}\leq \Omega^2 r^* \notag \\
 \leq &(B\rho)_r(r^*,\theta,z) + (\Phi_K)_r(r^*,\theta,z).\notag
\end{align}
The first term above is bounded by $2CC_1d$ by lemma \ref{chap3: lem: nonexist L infty bound on B rho_r} and \ref{chap3: lem: nonexistence l infinity bound on rho}. Noticing $(\Phi_K)_r(0,\theta,z)=0$ by axisymmetry, the second term above is therefore bounded by $Cd^{\frac{3}{q}}$ by lemma \ref{chap3: lem: Holder and Lip continuity of Phi_K}. Now we have
\begin{align}
\Omega^2 \frac{d}{2} \leq \tilde{C}(d+d^{\frac{3}{q}})\notag \\
(\frac{\Omega^2}{2}-\tilde{C})d^{1-\frac{3}{q}}  \leq \tilde{C}, \label{chap3: ineq: nonexist crucial ineq}
\end{align}
and the assertion follows.
\end{proof}

\begin{lemma}\label{chap3: lem: nonexist Holder continuity of rho on B}
$\rho\in C^{0,\alpha}(\bar{S})$ for some $0<\alpha<1$, where $S$ is any ball of radius $\frac{1}{2}$ whose center is on $(\mathbb{R}^3\setminus Z) \cap x_3$-axis, and we have $\|\rho\|_{C^{0,\alpha}(\bar{S})}\leq C_2$. Here $C_2$ is a constant depending on $\Phi_K$ $f$ and $M$, and $Z$ is the region given in lemma \ref{chap3: lem: Holder and Lip continuity of Phi_K}.
\end{lemma}
\begin{proof}
We first observe that since $A''(s)=\frac{f'(s)}{s}$, $A'(s)$ is strictly increasing. Also notice that $A'(0)=0$. By \eqref{chap3: eq: equilibrium eq}, $A'(\rho)$ is uniformly Lipschitz continuous on $S\cap \{\rho>0\}$ and continuous on $S$, hence is uniformly Lipschitz continuous on $S$. It is sufficient to prove 
\begin{equation*}
|\rho(\mathbf{x})-\rho(\mathbf{y})|\leq \tilde{C_2}|A'(\rho(\mathbf{x}))-A'(\rho(\mathbf{y}))|^{\alpha},
\end{equation*}
or
\begin{equation}
A'(t)-A'(s)\geq \tilde{C_2}(t-s)^{\frac{1}{\alpha}}
\end{equation}
for $0\leq s<t \leq \|\rho\|_{\infty}\leq C_1$. By \eqref{chap3: cond: f continuous differentiability} there exists a $C>0$ such that $f'(s)\geq Cs^{\mu}$ for $0\leq s \leq \|\rho\|_{\infty}\leq C_1$. Now let $u=t-s$, 
\begin{align}
& (A'(t)-A'(s))(t-s)^{-\frac{1}{\alpha}} \notag \\
= & (A'(s+u)-A'(s))u^{-\frac{1}{\alpha}} \notag \\
= & u^{-\frac{1}{\alpha}} \int_s^{s+u}A''(\xi)d\xi \notag \\
= & u^{-\frac{1}{\alpha}} \int_s^{s+u}\frac{f'(\xi)}{\xi}d\xi \notag \\
\geq & u^{-\frac{1}{\alpha}} \int_s^{s+u}\frac{C\xi ^{\mu}}{\xi}d\xi \notag \\
= &\tilde{C}u^{-\frac{1}{\alpha}}((s+u)^{\mu}-s^{\mu}) \label{chap3: eq: lower bound on (A'(s+u)-A'(s))u^ 1/alpha imdt step}
\end{align}
If $\mu \geq 1$, \eqref{chap3: eq: lower bound on (A'(s+u)-A'(s))u^ 1/alpha imdt step} is equal to
\begin{align}
& \tilde{C} \bigg[\bigg(1+\frac{s}{u}\bigg)^{\mu}-\bigg(\frac{s}{u}\bigg)^{\mu}\bigg]u^{\mu-\frac{1}{\alpha}} \notag \\
\geq &\tilde{C} u^{\mu-\frac{1}{\alpha}} \notag \\
\geq &\tilde{C} C_1^{\mu-\frac{1}{\alpha}}\geq \tilde{C_2}>0. \notag
\end{align}
The last step is correct if we choose $\alpha < \frac{1}{\mu}$. On the other hand if $0<\mu<1$, \eqref{chap3: eq: lower bound on (A'(s+u)-A'(s))u^ 1/alpha imdt step} is equal to
$\tilde{C}\mu \xi ^{\mu-1} u^{1-\frac{1}{\alpha}}$, where $\xi$ is between $s$ and $s+u$. This in turn is greater than or equal to 
\begin{align}
\tilde{C}\mu C_1^{\mu -1} u^{1-\frac{1}{\alpha}} \geq \tilde{C}\mu C_1^{\mu -1} C_1^{1-\frac{1}{\alpha}} \geq \tilde{C_2}>0\notag 
\end{align}
if we choose an $\alpha<1$.
\end{proof}

\begin{lemma}\label{chap3: lem: nonexist bound on supp rho in z direction}
There is an $\Omega_2>0$ such that if $\Omega>\Omega_2$, $\rho_{\mathbb{R}^3\setminus Z} \equiv 0$. Here $Z$ is the region given in lemma \ref{chap3: lem: Holder and Lip continuity of Phi_K}.
\end{lemma}
\begin{proof}
We first show that $\|B\rho\|_{C^{1,1}(\bar{S_1})}$ is uniformly bounded, where $S_1$ is any ball of radius $\frac{1}{4}$ whose center is on $(\mathbb{R}^3\setminus Z) \cap x_3$-axis. Let $S$ be a ball concentric with $S_1$ of radius $\frac{1}{2}$, then $B\rho= B(\rho\chi_{S})+B(\rho\chi_{\mathbb{R}^3 \setminus S})$. The first term is bounded in $C^{2,\alpha}(\bar{S_1})$ by lemma \ref{chap3: lem: nonexist Holder continuity of rho on B} and elliptic Schauder estimates. The second term is bounded in $C^2(\bar{S_1})$ by a direct differentiation under the integral sign argument since $\mathbb{R}^3 \setminus S$ is bounded away from $S_1$.

We first pick $\Omega>\Omega_1$ so that $d<d_0<\frac{1}{4}$. Now suppose $\rho(r,\theta,z)>0$ for some $(r,\theta,z)$ in $\mathbb{R}^3 \setminus Z$. Let us switch to Cartesian coordinates for the moment and, without loss of generality, denote this point $(x,0,z)$ with $x\geq 0$. Let $x^*=\sup\big\{x~\big|~ \rho(x,0,z)>0\big\}$. There must be a sequence $x_n\to x^*$ such that $\rho(x_n,0,z)>0$ and $(A'(\rho))_x(x_n,0,z)\leq 0$, differentiating \eqref{chap3: eq: equilibrium eq} with respect to $x$ and evaluating at $(x_n,0,z)$, we have
\begin{equation}
\Omega^2 x_n\leq -(B\rho)_x(x_n,0,z) - (\Phi_K)_x(x_n,0,z).
\end{equation}
Taking limit as $n \to \infty$, we get
\begin{equation}\label{chap3: ineq: local 1}
\Omega^2 x^*\leq -(B\rho)_x(x^*,0,z) - (\Phi_K)_x(x^*,0,z).
\end{equation}
If there was an $x_0 \in [0,x^*)$ such that $\rho(x_0,0,z)>0$ and $(A'(\rho))_x(x_0,0,z)\geq 0$, we would have
\begin{equation}\label{chap3: ineq: local 2}
\Omega^2 x_0\geq -(B\rho)_x(x_0,0,z) - (\Phi_K)_x(x_0,0,z).
\end{equation}
Subtracting \eqref{chap3: ineq: local 2} from \eqref{chap3: ineq: local 1}, we get
\begin{align}\label{chap3: ineq: local 3}
\Omega^2 (x^*-x_0) \leq \big((B\rho)_x(x_0,0,z)-(B\rho)_x(x^*,0,z)\big) \notag \\
+\big((\Phi_K)_x(x_0,0,z)- (\Phi_K)_x(x^*,0,z)\big).
\end{align}
The first term on the right hand side is bounded by $C(x^*-x_0)$ because $B\rho$ is uniformly bounded in $C^{1,1}(\bar{S_1})$ as indicated above, while the second term is bounded by $C(x^*-x_0)$ because of lemma \ref{chap3: lem: Holder and Lip continuity of Phi_K}. Hence \eqref{chap3: ineq: local 3} becomes
\begin{equation}
\Omega^2 (x^*-x_0) \leq 2C(x^*-x_0),
\end{equation}
which is impossible if we choose $\Omega_2 > \max\{\Omega_1, 2C\}$. Therefore such an $x_0$ does not exist. This in particular implies that there is no $x\in [0,x^*)$ for which $\rho(x,0,z)=0$, which then implies that $\rho(0,0,z)>0$ and $(A'(\rho))_x(0,0,z)<0$. But exactly the same argument in the $-x$ direction would imply $(A'(\rho))_x(0,0,z)>0$. This contradiction indicates that there is no such $(r,\theta,z)$ in the first place, and the assertion is therefore true.
\end{proof}

We are now ready to give
\begin{proof}[Proof of theorem \ref{chap3: thm: nonexistence of solution}]
By lemma \ref{chap3: lem: nonexist bound on supp rho in z direction}, $\bf{Supp}\rho$ is uniformly bounded in the $z$ direction. Recall from lemma \ref{chap3: lem: Holder and Lip continuity of Phi_K} that this bound is given by $l$, 
\begin{align}
M& =\int_{\bf{Supp} \rho}\rho \notag \\
& = \int_{|z|\leq l, r\leq d}\rho \notag\\
& \leq \|\rho\|_{\infty}2\pi d^2 l \notag\\
& \leq 2C_1\pi d^2 l.
\end{align}
Therefore 
\begin{equation}
d\geq \sqrt{\frac{M}{2C_1\pi l}}.
\end{equation}
Compare this with \eqref{chap3: ineq: nonexist crucial ineq}, we get
\begin{equation}
(\frac{\Omega^2}{2}-\tilde{C})\sqrt{\frac{M}{2C_1\pi l}}^{1-\frac{3}{q}}  \leq \tilde{C},
\end{equation}
which is clearly false if we choose $\Omega_0>\Omega_2$ sufficiently large. This contradiction indicates that such a solution $\rho$ does not exist.
\end{proof}



\bibliography{biblio}   

\end{document}